\documentclass[12pt]{amsart}
\usepackage{amsmath,amssymb}
\usepackage{amsfonts}
\usepackage{amsthm}
\usepackage{latexsym}
\usepackage{graphicx}

\def\supp{\mbox{supp}}

\def\p{\partial}

\def\vv<#1>{\langle#1\rangle}

\def\XXint#1#2{\setbox0=\hbox{$#1{#2}{\int}$}{#2}\kern-.5\wd0 }

\def\XXint#1#2#3{{\setbox0=\hbox{$#1{#2#3}{\int}$}
     \vcenter{\hbox{$#2#3$}}\kern-.5\wd0}}

\def\vv<#1>{\langle#1\rangle}
\def\e{\epsilon}
\newtheorem{thm}{Theorem}[section]

\newtheorem{lem}{Lemma}[section]

\newtheorem{cor}{Corollary}[section]
\theoremstyle{definition}

\theoremstyle{remark}

\numberwithin{equation}{section}

\begin{document}

\title{On the Li-Yau type gradient estimate of Li and Xu}
\author{Zhigang Chen}
\address{Department of Mathematics, Shantou University, Shantou, Guangdong, 515063, China}
\email{14zgchen@stu.edu.cn}
\author{Chengjie Yu$^1$}
\address{Department of Mathematics, Shantou University, Shantou, Guangdong, 515063, China}
\email{cjyu@stu.edu.cn}
\author{Feifei Zhao}
\address{Department of Mathematics, Shantou University, Shantou, Guangdong, 515063, China}
\email{14ffzhao@stu.edu.cn}

\thanks{$^1$Research partially supported by the Yangfan project from Guangdong Province and NSFC 11571215.}

\renewcommand{\subjclassname}{%
  \textup{2010} Mathematics Subject Classification}
\date{}
\begin{abstract}
In this paper, we obtain a Li-Yau type gradient estimate with time dependent parameter for positive solutions of the heat equation, so that the Li-Yau type gradient estimate of Li-Xu \cite{LX} are special cases of the estimate. We also obtain improvements of Davies' Li-Yau type gradient estimate. The argument is different with those of Li-Xu \cite{LX} and Qian \cite{Qi}.
\end{abstract}
\maketitle\markboth{Chen, Yu \& Zhao}{Li-Yau type gradient estimate}
\section{Introduction}
In recent years, Li and Xu \cite{LX} obtained the following Li-Yau type gradient estimate with time dependent parameter.
\begin{thm}\label{thm-LX}
Let $(M^n,g)$ be a complete Riemannian manifold with Ricci curvature bounded from below by $-k$, where $k$ is a nonnegative constant. Let $u\in C^\infty(M\times [0,T])$ be a positive solution of the heat equation
\begin{equation}\label{eq-heat}
\Delta u-u_t=0.
\end{equation}
Then,
\begin{equation}\label{eq-LX}
\|\nabla f\|^2-\left(1+\frac{\sinh(kt)\cosh(kt)-kt}{\sinh^{2}(kt)}\right)f_t\leq\frac{nk}{2}[\coth(kt)+1].
\end{equation}
and
\begin{equation}\label{eq-LX-linear}
\|\nabla f\|^2-\left(1+\frac{2}{3}kt\right)f_t\leq\frac{n}{2t}+\frac{nk}{2}\left(1+\frac{1}{3}kt\right)
\end{equation}
on $M\times (0,T]$, where $f=\log u$.
\end{thm}
 The estimate \eqref{eq-LX} and \eqref{eq-LX-linear}  of Li and Xu are of the same spirit of  Hamilton's Li-Yau type gradient estimate (see \cite{Ha}):
\begin{equation}\label{eq-Ha}
\|\nabla f\|^2-e^{2kt}f_t\leq e^{4kt}\frac{n}{2t}.
\end{equation}
We can compare the estimates \eqref{eq-LX}, \eqref{eq-LX-linear} and \eqref{eq-Ha} with the Li-Yau-Davies gradient estimate \cite{LY,Da}:
\begin{equation}\label{eq-LYD}
\|\nabla f\|^2-\alpha f_t\leq \frac{n\alpha^2}{2t}+\frac{n\alpha^2k}{4(\alpha-1)}
\end{equation}
for any $\alpha>1$ as follows. By comparing to the asymptotic behavior of the heat kernel as $t\to0$, we know that \eqref{eq-LYD} is even not sharp in  leading term. However, \eqref{eq-LX}, \eqref{eq-LX-linear} and \eqref{eq-Ha} are all sharp in leading term as $t\to 0$. When $t\to\infty$, it is clear that \eqref{eq-LYD} is better than \eqref{eq-LX-linear} and \eqref{eq-Ha}. For \eqref{eq-LX}, note that
\begin{equation}
1+\frac{\sinh(kt)\cosh(kt)-kt}{\sinh^{2}(kt)}\to 2 (\triangleq\alpha)
\end{equation}
and
\begin{equation}
\frac{nk}{2}[\coth(kt)+1]\to nk=\frac{n\alpha^2k}{4(\alpha-1)}
\end{equation}
as $t\to\infty$. So, the asymptotic behavior of \eqref{eq-LX} as $t\to\infty$ is the same as \eqref{eq-LYD} with $\alpha=2$. Note that \eqref{eq-LX-linear} was first obtained  in \cite{BQ} by a different method.

Theorem \ref{thm-LX} was later generalized by Qian \cite{Qi} to the following general form.
\begin{thm}\label{thm-Qian}
Let the notation be the same as in Theorem \ref{thm-LX} with $M$ closed. Then
\begin{equation}\label{eq-Qian}
\|\nabla f\|^2-\alpha f_t\leq\varphi
\end{equation}
with
\begin{equation}
  \alpha=1+\frac{2k}{a(t)}\int_0^ta(s)ds
\end{equation}
and
\begin{equation}
\varphi=\frac{nk}{2}+\frac{nk^{2}}{2a(t)}\int_{0}^{t}a(s)ds+\frac{n}{8a(t)}\int_{0}^{t}\frac{a'^{2}(s)}{a(s)}ds,
\end{equation}
where  $a\in C^1([0,T])$ is a smooth function satisfying:
\begin{enumerate}
 \item[(A1)]$\forall t\in (0,T]$, $a(t)>0$ and $a'(t)>0$;
 \item[(A2)]$a(0)=0$ and  $\lim_{t\to0}\frac{a(t)}{a'(t)}=0$;
 \item[(A3)] $\frac{a'^2}{a}\in L^1([0,T])$.
 \end{enumerate}
\end{thm}
For the complete noncompact case, some further technical conditions for the function $a$ should be satisfied. See \cite{Qi} for details.
 The estimates \eqref{eq-LX} and \eqref{eq-LX-linear} are special cases of \eqref{eq-Qian} with $a(t)=\sinh^2(kt)$ and $a(t)=t^2$ respectively. When, $a(t)=t^{\frac{2}{\theta}-1}$ with $\theta\in (0,1)$, one have
 \begin{equation}\label{eq-Qian-1}
\|\nabla f\|^2-(1+\theta kt) f_t\leq\frac{(2-\theta)^2n}{16\theta(1-\theta)t}+\frac{nk^2\theta t}{4}+\frac{nk}{2}
\end{equation}
for complete Riemannian manifolds with Ricci curvature bounded from below by $-k$ where $k$ is a positive constant.

In this paper, we first obtain the following Li-Yau type gradient estimate for closed manifolds.
\begin{thm}\label{thm-main}
Let the notation be the same as in Theorem \ref{thm-LX} with $M$ closed, and $\lambda, \beta, \psi\in C^1((0,T])$  such that
\begin{enumerate}
\item[(B1)] $0<\beta(t)<1$ for any $t\in (0,T]$;
\item[(B2)] $\lim_{t\to 0^+}\lambda(t)=0$ and $\lambda(t)>0$ for any $t\in (0,T]$;
\item[(B3)]  $\frac{2k\beta+\beta'}{1-\beta}-(\ln\lambda)'>0$ for any $t\in (0,T]$;
\item[(B4)]  $\limsup_{t\to0^+}\psi(t)\geq 0$;
\item[(B5)] $\psi'+\frac{2k\beta+\beta'}{1-\beta}\psi-\frac{n(2k\beta+\beta')^2}{8\beta(1-\beta)^2}=0$ for any $t\in (0,T]$.
\end{enumerate}
Then,
\begin{equation}\label{eq-main}
\beta\|\nabla f\|^2-f_t\leq \psi
\end{equation}
on $M\times (0,T]$.
\end{thm}

We write the Li-Yau type estimate in the form \eqref{eq-main} because it is more convenience for comparison. This form was also took in \cite{YZ,ZZ0,ZZ}. The Li-Yau-Davies estimate \eqref{eq-LYD} written in this form is:
\begin{equation}\label{eq-LYD2}
\beta\|\nabla f\|^2-f_t\leq \frac{n}{2\beta t}+\frac{nk}{4(1-\beta)}.
\end{equation}

Comparing Theorem \ref{thm-main} to Theorem \ref{thm-Qian}, we have $\beta=\frac{1}{\alpha}$. Let $\psi=\frac{1}{\alpha}\varphi$, by Lemma 2.2 in \cite{Qi}, we know that $\psi$ satisfies (B5). So, if we choose $\beta=\frac{1}{\alpha}$ satisfying (B1)--(B4), then Theorem \ref{thm-main} gives us the same conclusion of Theorem \ref{thm-Qian}.

As a corollary of Theorem \ref{thm-main}, we have the following general  Li-Yau type estimate with similar spirit to that of Theorem \ref{thm-Qian}.
\begin{cor}\label{cor-main}
Let the notation be the same as in Theorem \ref{thm-LX} with $M$ closed and $k>0$, and  $b\in C^1((0,T])$ satisfy that
\begin{enumerate}
\item[(C1)] $\lim_{t\to 0^+}b(t)=0$ and $b'(t)>0$ for any $t\in (0,T]$;
\item[(C2)]  $\frac{b'^2}{b}\in L^1([0,T])$;
\end{enumerate}
Then
\begin{equation}
  \beta\|\nabla f\|^2-f_t\leq \psi,
\end{equation}
where
\begin{equation}\label{eq-beta}
\beta=1-\frac{2k}{b(t)e^{2kt}}\int_0^tb(s)e^{2ks}ds
\end{equation}
and
\begin{equation}\label{eq-psi}
\psi=\frac{n}{8b}\int_0^t\frac{b'^2}{b\beta}(s)ds.
\end{equation}
\end{cor}
By direct computation, when
\begin{equation}
b(t)=(1+\theta k t)t^{\frac{2}\theta-1},
\end{equation}
one has
\begin{equation}
\beta(t)=\frac{1}{1+\theta kt}.
\end{equation}
Moreover, $b$ satisfies (C1) and (C2) if and only if $\theta\in (0,1)$. This gives us \eqref{eq-Qian-1} for closed manifolds.

When $b(t)=\sinh^2(kt)+\cosh(kt)\sinh(kt)-kt$, by direct computation,
\begin{equation}
\beta=\frac{1}{1+\frac{\sinh(kt)\cosh(kt)-kt}{\sinh^{2}(kt)}}.
\end{equation}
So, Corollary \ref{cor-main} also gives us \eqref{eq-LX} for closed manifolds.

Moreover, by setting $$b=a+2k\int_0^ta(s)ds$$ with $a$ in Theorem \ref{thm-Qian}, it not hard to verify that $b$ satisfies (C1) and (C2). By direct computation, one has
\begin{equation}
\beta=\frac{1}{1+\frac{2k}{a}\int_0^ta(s)ds},
\end{equation}
and Corollary \ref{cor-main} gives us Theorem \ref{thm-Qian}.

For the complete noncompact case, similar with that of \cite{Qi}, we have to add more restricted assumptions.
\begin{thm}\label{thm-main-noncom}
Let the notation be the same as in Theorem \ref{thm-LX} with $M$ complete noncompact, and $\lambda,\beta,\psi\in C^1((0,T])$  such that
\begin{enumerate}
\item[(B1')]  $\lim_{t\to 0^+}\beta(t)=1$ and $0<\beta(t)<1$ for any $t\in (0,T]$;
\item[(B2')] $\lim_{t\to 0^+}\lambda(t)=0$, $\lambda'(t)>0$  for any $t\in (0,T]$;
\item[(B$2\frac12$)] $\frac{\lambda}{1-\beta}$  and $\beta'$ are bounded from above on $(0,T]$;
\item[(B3')]  there is some $\e>0$ such that $\frac{2k\beta+\beta'}{1-\beta}-(1+\e)(\ln\lambda)'>0$ for any $t\in (0,T]$;
\item[(B4')]  $\psi(t)\geq 0$ for any $t\in (0,T]$;
\item[(B5)] $\psi'+\frac{2k\beta+\beta'}{1-\beta}\psi-\frac{n(2k\beta+\beta')^2}{8\beta(1-\beta)^2}=0$ for any $t\in (0,T]$.
\end{enumerate}
Then,
\begin{equation}
\beta\|\nabla f\|^2-f_t\leq \psi
\end{equation}
on $M\times (0,T]$.
\end{thm}

Furthermore, for complete noncompact Riemannian manifolds, one has the following similar corollary with more restricted assumptions.
\begin{cor}\label{cor-main-noncom}
Let the notation be the same as in Theorem \ref{thm-LX} with $M$ complete noncompact and $k>0$, and  $b\in C^1((0,T])$ satisfy that
\begin{enumerate}
\item[(C1)] $\lim_{t\to 0^+}b(t)=0$ and $b'(t)>0$ for any $t\in (0,T]$;
\item [(C2)] $\frac{b'^2}{b}\in L^1([0,T])$;
\item[(C3)] there is a constant $\delta\in (0,1)$ such that $\frac{b'}{b^\delta}$ is bounded from above on $(0,T]$;
\item [(C4)] $\frac{b'\int_0^tb(s)ds}{b^2}$ is bounded from above on $(0,T]$.
\end{enumerate}
Then
\begin{equation}
  \beta\|\nabla f\|^2-f_t\leq \psi,
\end{equation}
where $\beta$ and $\psi$ are given in \eqref{eq-beta} and \eqref{eq-psi} respectively.
\end{cor}
It is clear that $$b(t)=(1+\theta k t)t^{\frac{2}\theta-1}$$ and $$b(t)=\sinh^2(kt)+\cosh(kt)\sinh(kt)-kt$$ also satisfy (C3) and (C4). So Corollary \ref{cor-main-noncom} also gives us \eqref{eq-Qian-1} and Theorem \ref{thm-LX} for complete noncompact Riemannian manifolds.

Finally, by using \eqref{eq-Qian-1}, we are able to obtain an improvement of the Li-Yau-Davies estimate \eqref{eq-LYD2} and the Li-Yau type gradient estimate in \cite{YZ} for large time.
\begin{thm}\label{thm-im-LYD}
Let the notation be the same as in Theorem \ref{thm-LX} with $k>0$. Then, for any $\beta\in (0,1)$ and $t>\frac{1-\beta}{k\beta} $,
\begin{equation}\label{eq-im-LYD}
\beta\|\nabla f\|^2-f_t\leq \frac{n(1-\beta)}{16k( t-\frac{1-\beta}{k\beta})t}+\frac{nk}{4(1-\beta)}.
\end{equation}
\end{thm}
This estimate is clearly better than \eqref{eq-LYD} and the Li-Yau type estimate in \cite{YZ} when time is large. In fact, by direct computation, we have the following straight forward corollary.
\begin{cor}
Let the notation be the same as in Theorem \ref{thm-LX} with $k>0$. Then,
\begin{enumerate}
\item for any $\gamma>\frac{1-\beta}{16k}$ and $t>\frac{\gamma(1-\beta)}{k\beta\left(\gamma-\frac{1-\beta}{16k}\right)}$,
    \begin{equation}
    \beta\|\nabla f\|^2-f_t\leq \frac{\gamma n}{t^2}+\frac{nk}{4(1-\beta)};
    \end{equation}
\item for any $\gamma>0$ and $t>\frac{1-\beta}{16k\gamma}+\frac{1-\beta}{k\beta}$,
\begin{equation}
\beta\|\nabla f\|^2-f_t\leq \frac{\gamma n}{t}+\frac{nk}{4(1-\beta)};
\end{equation}
\item for any $\gamma>0$ and $\theta\in (1,2)$, there is a positive constant $T_0(k,\beta, \theta, \gamma)$ such that for any $t>T_0$,
    \begin{equation}
\beta\|\nabla f\|^2-f_t\leq \frac{\gamma n}{t^\theta}+\frac{nk}{4(1-\beta)}.
\end{equation}
\end{enumerate}
\end{cor}

Although Theorem \ref{thm-main} (or Theorem \ref{thm-main-noncom}) and Corollary \ref{cor-main} (or Corollary \ref{cor-main-noncom}) are similar to Li-Xu's estimate (Theorem \ref{thm-LX}) and Qian's generalization (Theorem \ref{thm-Qian}), the proofs of Theorem \ref{thm-main} (or Theorem \ref{thm-main-noncom}) and Corollary \ref{cor-main} (or Corollary \ref{cor-main-noncom}) are different to  those of Li-Xu \cite{LX} and Qian \cite{Qi}, where Li-Xu and Qian applied the maximum principle to
\begin{equation}
\left(\Delta-\frac{\p}{\p t}\right)(auF)=2au\left(\left\|\nabla^2f+\left(\frac{k}{2}+\frac{a'}{4a}\right)g\right\|^2+(Ric+kg)(\nabla f,\nabla f)\right)
\end{equation}
with $F=\|\nabla f\|^2-\alpha f_t-\varphi$ which is more in the spirit of Perelman \cite{Pe} (see also \cite{Ni1,Ni2}), while we simply apply the maximum principle to $\lambda(\beta\|\nabla f\|^2-f_t-\psi)$ which is similar to that of Li-Yau \cite{LY}.
\section{Li-Yau type gradient estimate}

We first prove Theorem \ref{thm-main}.
\begin{proof}[Proof of Theorem \ref{thm-main}]
Let $G=\beta\|\nabla f\|^2-f_t-\psi$, and $L=\Delta-\p_t$. Note that
\begin{equation}
Lf=-\|\nabla f\|^2,
\end{equation}
\begin{equation}
Lf_t=-2\vv<\nabla f_t,\nabla f>
\end{equation}
and
\begin{equation}
L\|\nabla f\|^2=2\|\nabla^2f\|^2+2Ric(\nabla f,\nabla f)-2\vv<\nabla\|\nabla f\|^2,\nabla f>.
\end{equation}
Then, by noting that
\begin{equation}
\|\nabla f\|^2=-\frac{1}{1-\beta}(\Delta f+G+\psi),
\end{equation}
we have
\begin{equation}\label{eq-LG-1}
\begin{split}
LG=&\beta L\|\nabla f\|^2-\beta'\|\nabla f\|^2-Lf_t+\psi'\\
\geq&\beta\left(\frac{2}{n}(\Delta f)^2-2k\|\nabla f\|^2-2\vv<\nabla\|\nabla f\|^2,\nabla f>\right)-\beta'\|\nabla f\|^2+2\vv<\nabla f_t,\nabla f>+\psi'\\
=&\frac{2\beta}{n}\left(\Delta f\right)^2-(2k\beta+\beta')\|\nabla f\|^2+\psi'-2\vv<\nabla G,\nabla f>\\
=&\frac{2\beta}{n}\left(\Delta f\right)^2+\frac{2k\beta+\beta'}{1-\beta}(\Delta f+G+\psi)+\psi'-2\vv<\nabla G,\nabla f>\\
\geq&\frac{2k\beta+\beta'}{1-\beta}G+\psi'+\frac{2k\beta+\beta'}{1-\beta}\psi-\frac{n(2k\beta+\beta')^2}{8\beta(1-\beta)^2}-2\vv<\nabla G,\nabla f >.\\
=&\frac{2k\beta+\beta'}{1-\beta}G-2\vv<\nabla G,\nabla f >,
\end{split}
\end{equation}
where we have used that $ax^2+bx\geq-\frac{b^2}{4a}$ when $a>0$.

Let $F=\lambda G$. Then,
\begin{equation}
\begin{split}
LF=&\lambda LG-(\ln\lambda)'F\geq\left(\frac{2k\beta+\beta'}{1-\beta}-(\ln\lambda)'\right)F-2\vv<\nabla F,\nabla f >.
\end{split}
\end{equation}
By (B1),(B2) and (B4), we know that
\begin{equation}
\liminf_{t\to0^+}F\leq 0.
\end{equation}
Then, by maximum principle, we complete the proof of the theorem.
\end{proof}
Next, we come to prove Corollary \ref{cor-main}.
\begin{proof}[Proof of Corollary \ref{cor-main}] By the expression \eqref{eq-beta}, it is clear that $\beta(t)<1$ for $t\in (0,T]$.  On the other hand, by integration by parts,
\begin{equation}
\beta(t)=\frac{1}{b(t)e^{2kt}}\int_0^tb'(s)e^{2ks}ds>0
\end{equation}
for $t\in (0,T]$. So, (B1) is satisfied. Let $\lambda=\sqrt b$, it is clear that (B2) is satisfied. Moreover, by direct computation, we have
\begin{equation}\label{eq-beta-b}
\frac{2k\beta+\beta'}{1-\beta}=(\ln b)'
\end{equation}
which implies that
\begin{equation}
\frac{2k\beta+\beta'}{1-\beta}-(\ln\lambda)'=\frac{1}{2}(\ln b)'>0
\end{equation}
for $t\in (0,T]$ by (C1). So (B3) is satisfied. (B4) is clearly true by expression \eqref{eq-psi} of $\psi$. Finally, (B5) can be verified by direct computation. So, by Theorem \ref{thm-main}, we complete the proof of the corollary.
\end{proof}
The proof of Theorem \ref{thm-main-noncom} uses the classical cut-off argument of Li-Yau \cite{LY}. First recall the following existence of cut-off functions.
\begin{lem}\label{lem-cutoff}
Let $(M^n,g)$ be a complete noncompact Riemannian manifold with Ricci curvature bounded from below. Then, there is a constant $C_1>1$ such that for any $p\in M$ and $R>1$, there is a smooth function  $\rho_R$ on $M$ satisfying:
\begin{enumerate}
\item $0\leq \rho_R\leq 1$;
\item $\rho_R|_{B_p(R)}\equiv 1$ and $\supp \rho_R\subset \overline{B_p(C_1R)}$;
\item $\|\nabla \rho_R\|^2\leq C_1R^{-2}\rho_R$ and $\Delta\rho_R\geq -C_1R^{-1}$ on $M$.
\end{enumerate}
\end{lem}
\begin{proof}
Let $r$ be a
smooth function on $M$, such that
\begin{equation}\label{eq-rho}
\left\{\begin{array}{l}C_{2}^{-1}(1+d(p,x))\leq r(x)\leq
C_2(1+d(p,x))\\\|\nabla r\|\leq C_2\\ |\Delta r|\leq C_2
\end{array}\right.
\end{equation}
all over $M$, where $C_2>1$ is some constant. The existence of such a function can be found in \cite{Yabook}. Let $\eta$ be a smooth function on $[0,+\infty)$ with (i) $\eta(t)=1$ for $t\in [0,1]$, (ii) $\eta(t)=0$ for $t\geq 2$ and (ii) $\eta'\leq 0$. Let $\rho_R(x)=\eta^2(\frac{r(x)}{2C_2R})$. It is not hard to check that $\rho_R$ satisfies the requirements of the lemma with $$C_1=\max\{4C_2^2,\max(|\eta'|+|\eta'|^2+|\eta''|)\}.$$
\end{proof}
We are now ready to prove Theorem \ref{thm-main-noncom}.
\begin{proof}
We will proceed by contradiction. Let $F$ and $G$ be the same as in the proof of Theorem \ref{thm-main}. Then, by \eqref{eq-LG-1}, we have
\begin{equation}\label{eq-LG-2}
\begin{split}
LG\geq&\frac{2\beta}{n}\left(\Delta f\right)^2-(2k\beta+\beta')\|\nabla f\|^2+\psi'-2\vv<\nabla G,\nabla f>\\
=&\frac{2\beta}{n}\left(\|\nabla f\|^2-f_t\right)^2-(2k\beta+\beta')\|\nabla f\|^2+\psi'-2\vv<\nabla G,\nabla f>\\
=&\frac{2\beta}{n}\left(G+\psi+(1-\beta)\|\nabla f\|^2\right)^2-(2k\beta+\beta')\|\nabla f\|^2+\psi'-2\vv<\nabla G,\nabla f>\\
=&\frac{2\beta}{n}(G+\psi+X)^2-\frac{2k\beta+\beta'}{1-\beta}(\psi+X)+\frac{n(2k\beta+\beta')^2}{8\beta(1-\beta)^2}-2\vv<\nabla G,\nabla f>\\
=&\frac{2\beta}{n}G^2+\frac{4\beta}{n}G(\psi+X)+\frac{2\beta}n\left(\psi+X-\frac{n(2k\beta+\beta')}{4\beta(1-\beta)}\right)^2-2\vv<\nabla G,\nabla f>\\
\end{split}
\end{equation}
where $X=(1-\beta)\|\nabla f\|^2$ and we have substituted (B5) into the inequality. Moreover,
\begin{equation}\label{eq-LF-2}
\begin{split}
&LF\\
=&\lambda LG-(\ln\lambda)'F\\
\geq&\frac{2\beta}{n\lambda}F^2+\frac{4\beta}{n}F(\psi+X)+\frac{2\lambda\beta}n\left(\psi+X-\frac{n(2k\beta+\beta')}{4\beta(1-\beta)}\right)^2-(\ln\lambda)'F\\
&-2\vv<\nabla F,\nabla f>.\\
\end{split}
\end{equation}

Suppose that $F(p,t_0)>0$ for some $p\in M$ and $t_0\in (0,T]$. For each $R>1$, let $\rho_R$ be the cut-off function in Lemma \ref{lem-cutoff}. Let $Q_R=\rho_RF$. Then, by \eqref{eq-LF-2},
\begin{equation}\label{eq-QR-1}
\begin{split}
LQ_R=&\rho_RLF+F\Delta\rho_R+2\vv<\nabla \rho_R,\nabla F>\\
\geq&\frac{2\beta}{n\lambda\rho_R}Q_R^2+\frac{4\beta}{n}Q_R(\psi+X)+\frac{2\lambda\rho_R\beta}n\left(\psi+X-\frac{n(2k\beta+\beta')}{4\beta(1-\beta)}\right)^2-(\ln\lambda)'Q_R\\
&-2\vv<\rho_R\nabla F,\nabla f>+F\Delta\rho_R+2\vv<\nabla \rho_R,\nabla F>.\\
\end{split}
\end{equation}
By (B1'), (B2') and (B4'), there is a $\bar t_R\in (0,t_0)$ small enough such that
\begin{equation}\label{eq-bar-t}
\max_{x\in M}Q_R(x,\bar t_R)<F(p,t_0)=Q_R(p,t_0).
\end{equation}
Let $(x_R,t_R)$ be the maximum point of $Q_R$ in $M\times [\bar t_R,T]$. By \eqref{eq-bar-t}, we have $t_R>\bar t_R$ and $Q_R(x_R,t_R)>0$. Then,
\begin{equation}\label{eq-g-F}
\nabla F(x_R,t_R)=-F(x_R,t_R)\rho_R^{-1}\nabla\rho_R(x_R)
\end{equation}
and
\begin{equation}\label{eq-QR-2}
0\geq LQ_R(x_R,t_R).
\end{equation}
So, by \eqref{eq-QR-1}, and multiplying $\lambda(t_R)\rho_R(x_R)$ to \eqref{eq-QR-2}, at the point $(x_R,t_R)$, we have
\begin{equation}\label{eq-QR-m}
\begin{split}
0\geq&\frac{2\beta}{n}Q_R^2+\frac{4\beta}{n}Q_R\lambda\rho_R(\psi+X)+\frac{2\beta\lambda^2\rho_R^2}n\left(\psi+X-\frac{n(2k\beta+\beta')}{4\beta(1-\beta)}\right)^2-\lambda'\rho_RQ_R\\
&-2\lambda\vv<\rho_R^2\nabla F,\nabla f>+\lambda Q_R\Delta\rho_R+2\lambda\vv<\nabla \rho_R,\rho_R\nabla F>.\\
=&\frac{2\beta}{n}Q_R^2+\frac{4\beta}{n}Q_R\lambda\rho_R(\psi+X)+\frac{2\beta\lambda^2\rho_R^2}n\left(\psi+X-\frac{n(2k\beta+\beta')}{4\beta(1-\beta)}\right)^2-\lambda'\rho_RQ_R\\
&+2\lambda Q_R\vv<\nabla \rho_R,\nabla f>+\lambda Q_R\Delta\rho_R-2\lambda\|\nabla \rho_R\|^2F\\
\geq&\frac{2\beta}{n}Q_R^2+\frac{4\beta}{n}Q_R\lambda\rho_R(\psi+X)+\frac{2\beta\lambda^2\rho_R^2}n\left(\psi+X-\frac{n(2k\beta+\beta')}{4\beta(1-\beta)}\right)^2-\lambda'\rho_RQ_R\\
&-2C_1R^{-1}\left(\frac{\lambda}{1-\beta}\right)^{\frac12}Q_R(\lambda\rho_RX)^\frac12-3C_1R^{-1}\lambda Q_R\\
\geq&\frac{2\beta}{n}Q_R^2+\frac{4\beta}{n}Q_R\lambda\rho_R(\psi+X)+\frac{2\beta\lambda^2\rho_R^2}n\left(\psi+X-\frac{n(2k\beta+\beta')}{4\beta(1-\beta)}\right)^2-\lambda'\rho_RQ_R\\
&-2C_3R^{-1}Q_R(\lambda\rho_RX)^\frac12-C_3R^{-1}Q_R\\
\geq&\frac{2\beta}{n}Q_R^2+\left(\frac{4\beta}{n}-C_3R^{-1}\right)Q_R\lambda\rho_R(\psi+X)+\frac{2\beta\lambda^2\rho_R^2}n\left(\psi+X-\frac{n(2k\beta+\beta')}{4\beta(1-\beta)}\right)^2\\
&-\lambda'\rho_RQ_R-2C_3R^{-1}Q_R.\\
\end{split}
\end{equation}
where we have used Lemma \ref{lem-cutoff}, (B$2\frac12$), (B4') and the fact $$2Q_R(\lambda\rho_RX)^\frac{1}2\leq Q_R+Q_R(\lambda\rho_R X).$$
Here $C_3=C_1\max\left\{\sup_{(0,T]} \sqrt\frac{\lambda}{1-\beta}, 3\max_{[0,T]}\lambda\right\}$.

Next, we divide the proof into three cases to draw a contradiction.
\begin{enumerate}
\item There is a sequence $R_i\to+\infty$ as $i\to\infty$, such that $\rho_{R_i}(x_{R_i})\to 0$ as $i\to \infty$. Then, by \eqref{eq-QR-m}, we have
\begin{equation}\label{eq-QR-3}
0<F(p,t_0)\leq Q_{R_i}(x_{R_i},t_{R_i})\leq \frac{n}{2\beta(t_{R_i})}\left(\lambda'(t_{R_i})\rho_{R_i}(x_{R_i})+2C_3R_i^{-1}\right)
\end{equation}
when $i$ is sufficiently large. By (B1'), (B$2\frac12$) and (B3'), we know that $\min_{[0,T]}\beta(t)>0$ and $\lambda'$ is bounded from above on $(0,T]$. So, taking $i\to\infty$ in \eqref{eq-QR-3} gives us  a contradiction.
\item There is a sequence $R_i\to+\infty$ as $i\to\infty$, such that $\lambda'(t_{R_i})\to 0$ as $i\to\infty$. Then, similarly as in case (1), by \eqref{eq-QR-3}, we can draw a contradiction.
\item If there is a positive constant $\e_0>0$ such that $\rho_{R}(x_R)\geq \e_0$ and $\lambda'(t_{R})\geq \e_0$ when $R$ is sufficiently large. Then, by \eqref{eq-QR-m},
at the point $(x_R,t_R)$,
\begin{equation}
\begin{split}
0\geq&\frac{2\beta}{n}Q_R^2+\left(\frac{4\beta}{n}-C_3R^{-1}\right)Q_R\lambda\rho_R(\psi+X)+\frac{2\beta\lambda^2\rho_R^2}n\left(\psi+X-\frac{n(2k\beta+\beta')}{4\beta(1-\beta)}\right)^2\\
&-(1+C_4R^{-1})\lambda'\rho_RQ_R\\
\geq&\frac{2\beta}{n}Q_R^2+\left(\frac{4\beta}{n}-C_3R^{-1}\right)Q_R\lambda\rho_R\left(\psi+X-\frac{n(2k\beta+\beta')}{4\beta(1-\beta)}\right)\\
&+\frac{2\beta\lambda^2\rho_R^2}n\left(\psi+X-\frac{n(2k\beta+\beta')}{4\beta(1-\beta)}\right)^2+(1-C_5R^{-1})\frac{2k\beta+\beta'}{1-\beta}\lambda\rho_RQ_R\\
&-(1+C_4R^{-1})\lambda'\rho_RQ_R\\
\geq&(1-C_5R^{-1})\lambda\left(\frac{2k\beta+\beta'}{1-\beta}-\frac{1+C_4R^{-1}}{1-C_5R^{-1}}(\ln\lambda)'\right)\rho_RQ_R.\\
\end{split}
\end{equation}
where $C_4=\frac{2C_3}{\e_0^2}$ and  $C_5=\frac{nC_3}{4\min_{[0,T]}\beta}$. Moreover, by (B2') and (B3), when $R$ is sufficiently large,
\begin{equation}
\frac{2k\beta+\beta'}{1-\beta}-\frac{1+C_4R^{-1}}{1-C_5R^{-1}}(\ln\lambda)'>0.
\end{equation}
So,
\begin{equation}
0<F(p,t_0)\leq Q_R(x_R,t_R)\leq 0
\end{equation}
when $R$ is sufficiently large. This is a contradiction.
\end{enumerate}
This completes the proof of the theorem.
\end{proof}

We next come to prove Corollary \ref{cor-main-noncom}.
\begin{proof}[Proof of Corollary \ref{cor-main-noncom}] Note that
\begin{equation}
\beta(t)=1-\frac{2k\int_0^tb(s)e^{2ks}ds}{b(t)e^{2kt}}\geq 1-2kt,
\end{equation}
by (C1), and it has been shown in the proof of Corollary \ref{cor-main} that $0<\beta(t)<1$ for $t\in (0,T]$. So (B1') is satisfied.

Let $\lambda=b^{1-\delta}$. Then, (B2') is satisfied by (C1).
By \eqref{eq-beta-b},
\begin{equation}
\frac{2k\beta+\beta'}{1-\beta}-(1+\delta)(\ln\lambda)'=\delta^2(\ln b(t))'>0
\end{equation}
for $t\in (0,T]$. So, (B3') is also satisfied. Moreover,
\begin{equation}
\begin{split}
\frac{\lambda}{1-\beta}=&\frac{b^{2-\delta}e^{2kt}}{2k\int_0^tb(s)e^{2ks}ds}\\
\leq& \frac1{2k}e^{2kT}\frac{b^{2-\delta}}{\int_0^tb(s)ds}\\
=&\frac{2-\delta}{2k}e^{2kT}\frac{\int_0^tb^{1-\delta}b'(s)ds}{\int_0^tb(s)ds}
\end{split}
\end{equation}
is bounded from above by (C3), and by \eqref{eq-beta-b},
\begin{equation}
\beta'\leq (1-\beta)(\ln b)'=\frac{2kb'\int_0^tb(s)e^{2ks}ds}{b^2e^{2kt}}\leq \frac{2kb'\int_0^tb(s)ds}{b^2}
\end{equation}
is bounded from above by (C4). So (B$2\frac12$) is satisfied.

Finally, by the expression \eqref{eq-psi} and direct computation, (B4') and (B5) is clearly satisfied. So,  by Theorem \ref{thm-main-noncom}, we complete the proof of the Corollary.
\end{proof}
Finally, we come to prove Theorem \ref{thm-im-LYD}.
\begin{proof}[Proof of Theorem \ref{thm-im-LYD}]
For each $\beta_0\in (0,1)$ and $t_0\geq \frac{1-\beta_0}{k\beta_0}$, let $\theta_0=\frac{1-\beta_0}{k\beta_0t_0}\in (0,1)$. Then $\beta_0=\frac{1}{1+\theta_0kt_0}$. In \eqref{eq-Qian-1}, let $\beta(t)=\frac1{1+\theta_0 kt}$ and $t=t_0$, we have
\begin{equation}
\begin{split}
&\beta_0\|\nabla f\|^2-f_t\\
\leq&\frac{(2-\theta_0)^2n\beta_0}{16\theta_0(1-\theta_0)t_0}+\frac{nk^2\beta_0\theta_0 t_0}{4}+\frac{nk\beta_0}{2}\\
=&\frac{n(2k\beta_0t_0+\beta_0-1)^2\beta_0}{16(1-\beta_0)(k\beta_0t_0+\beta_0-1)t_0}+\frac{nk(\beta_0+1)}{4}\\
=&\frac{n(1-\beta_0)}{16k( t_0-\frac{1-\beta_0}{k\beta_0})t_0}+\frac{nk}{4(1-\beta_0)}.
\end{split}
\end{equation}
This completes the proof of the theorem.
\end{proof}

\end{document}